\documentclass{article}

\usepackage{amsmath}
\usepackage{amsthm}
\usepackage{alltt}
\usepackage{hyperref}

\newtheorem{theorem}{Theorem}[section]

\newtheorem{definition}{Definition}[section]

\numberwithin{equation}{section}

\newcommand{\arccot}
{\textrm{\rm arccot}}
\newcommand{\arctanh}
{\textrm{\rm arctanh}}
\newcommand{\arccoth}
{\textrm{\rm arccoth}}
\newcommand{\sg}
{\textrm{\rm sg}}
\newcommand{\re}
{\textrm{\rm Re}}
\newcommand{\im}
{\textrm{\rm Im}}

\title{Higher Derivatives of the Tangent and Inverse Tangent Functions and Chebyshev Polynomials}

\author{M.J. Kronenburg}
\date{}

\begin{document}

\maketitle

\begin{abstract}
The higher derivatives of the tangent and hyperbolic tangent functions are determined.
Formulas for the higher derivatives of the inverse tangent and inverse hyperbolic tangent
functions as polynomials are stated and proved. 
Using another formula for the higher derivatives of the inverse tangent function from literature,
two known formulas for the Chebyshev polynomials of the first and second kind are proved.
From these formulas the higher derivatives of the inverse tangent
and inverse hyperbolic tangent functions in terms of the Chebyshev polynomial of the second kind are provided.
\end{abstract}

\noindent
\textbf{Keywords}: higher derivatives, inverse tangent function, chebyshev polynomials.\\
\textbf{MSC 2010}: 33B10, 33C45.

\section{Higher Derivatives of the Tangent and\\ Hyperbolic Tangent Functions}

The higher derivatives of the tangent and hyperbolic tangent functions
are computed in the following way \cite{KB67}. The first derivative of the tangent function is:
\begin{equation}
 D_x \tan(x) = D_x \frac{\sin(x)}{\cos(x)} = \frac{\cos^2(x)+\sin^2(x)}{\cos^2(x)} = 1 + \tan^2(x)
\end{equation}
By repeated application of this formula polynomials in $\tan(x)$ are obtained \cite{KB67}:
\begin{equation}
 D_x \tan^k(x) = k \tan^{k-1}(x) ( 1 + \tan^2(x) ) = k ( \tan^{k-1}(x) + \tan^{k+1}(x) )
\end{equation}
From this it is clear that the coefficients $T_{n,k}$ in the polynomial in $\tan(x)$:
\begin{equation}
 D_x^n \tan(x) = \sum_{k=0}^{n+1} T_{n,k} \tan^k(x)
\end{equation}
have the recursion relation \cite{KB67}:
\begin{equation}
 T_{n,k} = (k-1) T_{n-1,k-1} + (k+1) T_{n-1,k+1}
\end{equation}
with boundery conditions $T_{0,k}=\delta_{k,1}$ and $T_{n,-1}=0$.
Special cases are $T_{n,n+1}=n!$ and the $T_{2n+1,0}$ are the tangent numbers \cite{KB67}:
\begin{equation}
 \tan(x) = \sum_{k=0}^{\infty} \frac{T_{2k+1,0}}{(2k+1)!} x^{2k+1}
\end{equation}
For these coefficients $T_{n,k}=0$ when $n-k+1$ is odd, and therefore $k$ can be replaced
with $n-2k+1$, where $n-2k+1=0$ is reached when $k=(n+1)/2$:\\
For integer $n\geq 0$:
\begin{equation}
 D_x^n \tan(x) = \sum_{k=0}^{\lfloor\frac{n+1}{2}\rfloor} T_{n,n-2k+1} \tan^{n-2k+1}(x)
\end{equation}
The same reasoning can be applied to the following functions:
\begin{equation}
 D_x \cot(x) = D_x \frac{\cos(x)}{\sin(x)} = \frac{-\sin^2(x)-\cos^2(x)}{\sin^2(x)} = - ( 1 + \cot^2(x) )
\end{equation}
For integer $n\geq 0$:
\begin{equation}
 D_x^n \cot(x) = (-1)^n \sum_{k=0}^{\lfloor\frac{n+1}{2}\rfloor} T_{n,n-2k+1} \cot^{n-2k+1}(x)
\end{equation}
\begin{equation}
 D_x \tanh(x) = D_x \frac{\sinh(x)}{\cosh(x)} = \frac{\cosh^2(x)-\sinh^2(x)}{\cosh^2(x)} = 1 - \tanh^2(x)
\end{equation}
For integer $n\geq 0$:
\begin{equation}
 D_x^n \tanh(x) = (-1)^n \sum_{k=0}^{\lfloor\frac{n+1}{2}\rfloor} (-1)^k T_{n,n-2k+1} \tanh^{n-2k+1}(x)
\end{equation}
\begin{equation}
 D_x \coth(x) = D_x \frac{\cosh(x)}{\sinh(x)} = \frac{\sinh^2(x)-\cosh^2(x)}{\sinh^2(x)} = 1 - \coth^2(x)
\end{equation}
For integer $n\geq 0$:
\begin{equation}
 D_x^n \coth(x) = (-1)^n \sum_{k=0}^{\lfloor\frac{n+1}{2}\rfloor} (-1)^k T_{n,n-2k+1} \coth^{n-2k+1}(x)
\end{equation}
The corresponding Mathematica\textsuperscript{\textregistered} \cite{W03} program:
\begin{alltt}
$RecursionLimit=Infinity;
T[0,k_]=KroneckerDelta[k,1];
T[n_,-1]=0;
T[n_,k_]:=T[n,k]=If[k>n+1,0,(k-1)T[n-1,k-1]+(k+1)T[n-1,k+1]]
DTan[n_]:=Sum[T[n,n-2k+1]Tan[x]^(n-2k+1),\{k,0,Floor[(n+1)/2]\}]
DCot[n_]:=(-1)^n Sum[T[n,n-2k+1]Cot[x]^(n-2k+1),\{k,0,Floor[(n+1)/2]\}]
DTanh[n_]:=(-1)^n Sum[(-1)^k T[n,n-2k+1]Tanh[x]^(n-2k+1),
 \{k,0,Floor[(n+1)/2]\}]
DCoth[n_]:=(-1)^n Sum[(-1)^k T[n,n-2k+1]Coth[x]^(n-2k+1),
 \{k,0,Floor[(n+1)/2]\}]
\end{alltt}

\section{Higher Derivatives of the Inverse Tangent and\\ Inverse Hyperbolic Tangent Functions}

\begin{theorem}\label{tander}
For integer $n\geq 1$:
\begin{equation}
 D_x^n \arctan(x) = \frac{(-1)^{n+1}(n-1)!}{(1+x^2)^n} \sum_{k=0}^{\lfloor\frac{n-1}{2}\rfloor} \binom{n}{2k+1}(-1)^k x^{n-2k-1}
\end{equation}
\end{theorem}
\begin{proof}
The following is a definition of the $\arctan(x)$ function \cite{AS72}:
\begin{equation}
\begin{split}
 \arctan(x) & = - \frac{i}{2} \ln( \frac{1+ix}{1-ix} ) \\
 & = - \frac{i}{2} \ln( \frac{i-x}{i+x} ) \\
 & = \frac{i}{2} [ \ln(i+x) - \ln(i-x) ] \\
\end{split}
\end{equation}
The derivatives of the $\ln(x)$ function are:
\begin{equation}\label{lnder}
 D_x^n \ln(x) = (-1)^{n+1} (n-1)! x^{-n}
\end{equation}
so the derivatives of the $\arctan(x)$ function are:
\begin{equation}\label{derarctan}
\begin{split}
 D_x^n \arctan(x) & = (-1)^{n+1} (n-1)! \frac{i}{2} [ \frac{1}{(i+x)^n} - (-1)^n \frac{1}{(i-x)^n} ] \\
 & = (-1)^n (n-1)! \frac{i}{2} [ \frac{1}{(x-i)^n} - \frac{1}{(x+i)^n} ] \\
\end{split}
\end{equation}
The complex expression in this equation can be evaluated with the binomial theorem:
\begin{equation}
\begin{split}
 & \frac{i}{2} [ \frac{1}{(x-i)^n} - \frac{1}{(x+i)^n} ] \\
 = & \frac{i}{2} \frac{(x+i)^n - (x-i)^n}{(1+x^2)^n} \\
 = & \frac{1}{(1+x^2)^n} \frac{i}{2} [ \sum_{k=0}^n \binom{n}{k} i^k x^{n-k} - \sum_{k=0}^n \binom{n}{k} (-1)^k i^k x^{n-k} ] \\
 = & \frac{1}{(1+x^2)^n} i \sum_{k=0}^n \binom{n}{k} \frac{1}{2}(1-(-1)^k) i^k x^{n-k} \\
\end{split}
\end{equation}
The summand in this expression is only nonzero when $k$ is odd:
\begin{equation}
 \frac{1}{2} ( 1 - (-1)^k ) = 
 \begin{cases}
   1 & \text{if $k$ is odd} \\
   0 & \text{if $k$ is even} \\
 \end{cases}
\end{equation}
Therefore $k$ can be replaced by $2k+1$, where the upper limit is reached when $2k+1=n$,
which is when $k=(n-1)/2$, which results in:
\begin{equation}
\begin{split}
  & \frac{1}{(1+x^2)^n} i \sum_{k=0}^{\lfloor\frac{n-1}{2}\rfloor} \binom{n}{2k+1} i^{2k+1} x^{n-2k-1} \\
 = & \frac{-1}{(1+x^2)^n} \sum_{k=0}^{\lfloor\frac{n-1}{2}\rfloor} \binom{n}{2k+1} (-1)^k x^{n-2k-1} \\
\end{split}
\end{equation}
and the theorem is proved.
\end{proof}
\begin{theorem}
For integer $n\geq 1$:
\begin{equation}\label{arctanhres}
 D_x^n \arctanh(x) = \frac{(n-1)!}{(1-x^2)^n} \sum_{k=0}^{\lfloor\frac{n-1}{2}\rfloor} \binom{n}{2k+1} x^{n-2k-1}
\end{equation}
\end{theorem}
\begin{proof}
A similar derivation using \cite{AS72}:
\begin{equation}
 \arctanh(x) = \frac{1}{2} \ln(\frac{1+x}{1-x})
\end{equation}
gives:
\begin{equation}
 D_x^n \arctanh(x) = (-1)^n (n-1)! \frac{1}{2} [ \frac{1}{(x-1)^n} - \frac{1}{(x+1)^n} ]
\end{equation}
A similar derivation as in the previous theorem gives this theorem.
\end{proof}
\begin{theorem}\label{dtancot}
For integer $n\geq 1$:
\begin{equation}
 D_x^n \arccot(x) = - D_x^n \arctan(x)
\end{equation}
\end{theorem}
\begin{proof}
\begin{equation}
 \arctan(x) = - \frac{i}{2} \ln( \frac{1+ix}{1-ix} )
\end{equation}
\begin{equation}
 \arccot(x) = \arctan(\frac{1}{x}) = - \frac{i}{2} \ln( \frac{x+i}{x-i} ) 
 = \frac{i}{2} [ \ln(x-i) - \ln(x+i) ]
\end{equation}
Using identity (\ref{lnder}) from theorem \ref{tander}:
\begin{equation}
 D_x^n \arccot(x) = (-1)^{n+1} (n-1)! \frac{i}{2} [ \frac{1}{(x-i)^n} - \frac{1}{(x+i)^n} ]
\end{equation}
Comparing this with (\ref{derarctan}) from theorem \ref{tander}, this theorem is proved.
\end{proof}
\begin{theorem}
For integer $n\geq 1$:
\begin{equation}
 D_x^n \arccoth(x) = D_x^n \arctanh(x)
\end{equation}
\end{theorem}
\begin{proof}
\begin{equation}
 \arctanh(x) = \frac{1}{2} \ln( \frac{1+x}{1-x} ) = -i \, \arctan(ix)
\end{equation}
\begin{equation}
 \arccoth(x) = \arctanh(\frac{1}{x}) = \frac{1}{2} \ln( \frac{x+1}{x-1} ) = i \, \arccot(ix)
\end{equation}
Using the previous theorem:
\begin{equation}
\begin{split}
 D_x^n \arccoth(x) & = i D_x^n \arccot(ix) \\
 & = -i D_x^n \arctan(ix) \\
 & = D_x^n \arctanh(x) \\ 
\end{split}
\end{equation}
\end{proof}

\section{Chebyshev Polynomials}

Let the following definition be given:
\begin{definition}\label{sgdef}
For real $x$:
\begin{equation}
 \sg(x) = 
 \begin{cases}
   1 & \text{\rm if $x\geq 0$} \\
   -1 & \text{\rm if $x<0$} \\
 \end{cases}
\end{equation}
\end{definition}
For the definition for complex $x$, see section 5.
\begin{theorem}\label{sinarcsin}
For integer $n\geq 0$:
\begin{equation}
 \sin(n \arcsin(x)) = x \sum_{k=0}^{\lfloor\frac{n-1}{2}\rfloor} \binom{n}{2k+1} (-1)^k x^{2k} (1-x^2)^{\frac{n-1}{2}-k}
\end{equation}
\end{theorem}
\begin{proof}
There is another expression of the higher derivatives of the inverse tangent function from literature \cite{AL16,VL10,S50},
which is also proved in the next section:
\begin{equation}\label{adegoke}
 D_x^n \arctan(x) = \frac{(-1)^{n+1}(n-1)!\sg^{n-1}(x)}{(1+x^2)^{n/2}} \sin(n \arcsin(\frac{1}{\sqrt{1+x^2}}))
\end{equation}
where $\sg(x)$ is given by definition \ref{sgdef}.
Equating this identity with theorem \ref{tander}:
\begin{equation}
 \sin( n \arcsin(\frac{1}{\sqrt{1+x^2}}) ) = \frac{\sg^{n-1}(x)}{(1+x^2)^{n/2}} \sum_{k=0}^{\lfloor\frac{n-1}{2}\rfloor} \binom{n}{2k+1} (-1)^k x^{n-2k-1}
\end{equation}
For $0\leq x\leq 1$, replacing $x$ with $\sqrt{1-x^2}/x$ gives the theorem.
For $-1\leq x<0$, because $\sin(n\arcsin(-x))=-\sin(n\arcsin(x))$ the identity remains valid,
and therefore the theorem is proved.
\end{proof}
\begin{theorem}\label{unpoly}
\begin{equation}
 U_n(x) = \sum_{k=0}^{\lfloor\frac{n}{2}\rfloor} \binom{n+1}{2k+1} (x^2-1)^k x^{n-2k}
\end{equation}
\end{theorem}
\begin{proof}
The definition of the Chebyshev polynomial of the second kind \cite{B08}:
\begin{equation}\label{undef}
 U_n(x) = \frac{1}{\sqrt{1-x^2}} \sin((n+1)\arccos(x))
\end{equation}
For $0\leq x\leq 1$, replacing $x$ with $\sqrt{1-x^2}$ and using $\arccos(\sqrt{1-x^2})=\arcsin(x)$:
\begin{equation}
 U_n(\sqrt{1-x^2}) = \frac{1}{x} \sin((n+1)\arcsin(x)) 
\end{equation}
and substituting the result of the previous theorem:
\begin{equation}
 U_n(\sqrt{1-x^2}) = \sum_{k=0}^{\lfloor\frac{n}{2}\rfloor} \binom{n+1}{2k+1} (-1)^k x^{2k} (1-x^2)^{\frac{n}{2}-k}
\end{equation}
Replacing $x$ with $\sqrt{1-x^2}$ gives the theorem.
For $-1\leq x<0$, using $\arccos(-x)=\pi-\arccos(x)$ and $\sin(\alpha+n\pi)=(-1)^n\sin(\alpha)$, from the
definition (\ref{undef}) follows:
\begin{equation}
 U_n(-x) = (-1)^n U_n(x)
\end{equation}
This also holds for the right side of the theorem, and therefore the theorem is proved.
\end{proof}
This formula is a known expression for the Chebyshev polynomial of the second kind \cite{W091},
which is now proved via the higher derivatives of the inverse tangent function.
From the formula for the Chebyshev polynomial of the second kind, the formula for the
Chebyshev polynomial of the first kind can be derived.
\begin{theorem}
\begin{equation}
 T_n(x) = \sum_{k=0}^{\lfloor\frac{n}{2}\rfloor} \binom{n}{2k} (x^2-1)^k x^{n-2k}
\end{equation}
\end{theorem}
\begin{proof}
The definition of the Chebyshev polynomial of the second kind \cite{B08}:
\begin{equation}
 U_n(x) = \frac{1}{\sqrt{1-x^2}} \sin( (n+1)\arccos(x) )
\end{equation}
Using the trigonometric identity:
\begin{equation}
 \sin(\alpha+\beta) = \sin(\alpha)\cos(\beta)+\cos(\alpha)\sin(\beta)
\end{equation}
and using $\cos(\arccos(x))=x$ and $\sin(\arccos(x))=\sqrt{1-x^2}$:
\begin{equation}
\begin{split}
 U_n(x) & = \frac{x}{\sqrt{1-x^2}} \sin( n \arccos(x) ) + \cos( n \arccos(x) ) \\
 & = x U_{n-1}(x) + T_n(x) \\
\end{split}
\end{equation}
which results in:
\begin{equation}
 T_n(x) = U_n(x) - x U_{n-1}(x)
\end{equation}
Using this, theorem \ref{unpoly} and:
\begin{equation}
 \binom{n+1}{2k+1} - \binom{n}{2k+1} = \binom{n}{2k}
\end{equation}
gives this theorem.
\end{proof}
\begin{theorem}\label{dtancheb}
For integer $n\geq 1$:
\begin{equation}
 D_x^n \arctan(x) = \frac{(-1)^{n+1}(n-1)!}{(1+x^2)^{\frac{n+1}{2}}} U_{n-1}(\frac{x}{\sqrt{1+x^2}})
\end{equation}
\end{theorem}
\begin{proof}
The following has been proved above:
\begin{equation}
 U_n(x) = \sum_{k=0}^{\lfloor\frac{n}{2}\rfloor} \binom{n+1}{2k+1} (x^2-1)^k x^{n-2k}
\end{equation}
Replacing $x$ with $x/\sqrt{1+x^2}$ gives:
\begin{equation}
 U_n(\frac{x}{\sqrt{1+x^2}}) = \frac{1}{(1+x^2)^{\frac{n}{2}}} \sum_{k=0}^{\lfloor\frac{n}{2}\rfloor} \binom{n+1}{2k+1} (-1)^k x^{n-2k}
\end{equation}
Replacing $n$ by $n-1$ and using theorem \ref{tander} gives this theorem.
\end{proof}
\begin{theorem}\label{dtanhcheb}
For integer $n\geq 1$:
\begin{equation}
 D_x^n \arctanh(x) = \frac{(-1)^{n+1}(n-1)!i^{n-1}}{(1-x^2)^{\frac{n+1}{2}}} U_{n-1}(\frac{ix}{\sqrt{1-x^2}})
\end{equation}
\end{theorem}
\begin{proof}
From the definitions:
\begin{equation}
 \arctan(x) = - \frac{i}{2} \ln( \frac{1+ix}{1-ix} )
\end{equation}
\begin{equation}
 \arctanh(x) = \frac{1}{2} \ln( \frac{1+x}{1-x} )
\end{equation}
follows:
\begin{equation}
 \arctanh(x) = - i \arctan(ix)
\end{equation}
Let:
\begin{equation}
 D_x^n \arctan(x) = F(x)
\end{equation}
Then:
\begin{equation}
 D_x^n \arctanh(x) = - i D_x^n \arctan(ix) = - i^{n+1} F(ix) = i^{n-1} F(ix)
\end{equation}
which with the previous theorem gives this result.
\end{proof}

\section{A New Proof of Formula (\ref{adegoke})}

In addition to the proofs in literature \cite{AL16,VL10,S50}, a new different proof is provided for formula (\ref{adegoke}):
\begin{theorem}\label{dtansin}
For integer $n\geq 1$:
\begin{equation}
 D_x^n \arctan(x) = \frac{(-1)^{n+1}(n-1)!\sg^{n-1}(x)}{(1+x^2)^{n/2}} \sin(n \arcsin(\frac{1}{\sqrt{1+x^2}}))
\end{equation}
where $\sg(x)$ is given by definition \ref{sgdef}.
\end{theorem}
\begin{proof}
First the theorem is proved for $x\neq 0$.
The theorem is true for $n=1$, and using induction on $n$, when the theorem is true for $n$,
it is true for $n+1$ when:
\begin{equation}\label{induc}
 D_x  \frac{\sin(n\arcsin(\frac{1}{\sqrt{1+x^2}}))}{(1+x^2)^{\frac{n}{2}}} = \frac{-n\,\sg(x)\sin((n+1)\arcsin(\frac{1}{\sqrt{1+x^2}}))}{(1+x^2)^{\frac{n+1}{2}}}
\end{equation}
The left side of (\ref{induc}) is evaluated with the product rule for differentiation,
and the right side of (\ref{induc}) with a trigonometric identity.
For the left side, using:
\begin{equation}
 D_x (1+x^2)^{-\frac{n}{2}} = -n x (1+x^2)^{-\frac{n}{2}-1}
\end{equation}
\begin{equation}
 D_x \arcsin(x) = \frac{1}{\sqrt{1-x^2}}
\end{equation}
\begin{equation}
 D_x \arcsin(\frac{1}{\sqrt{1+x^2}}) = - \frac{\sqrt{1+x^2}}{|x|} x (1+x^2)^{-\frac{3}{2}} = - \frac{\sg(x)}{1+x^2}
\end{equation}
the left side yields:
\begin{equation}
 \frac{-n}{(1+x^2)^{\frac{n}{2}+1}} [ x \sin(n\arcsin(\frac{1}{\sqrt{1+x^2}})) + \sg(x) \cos(n\arcsin(\frac{1}{\sqrt{1+x^2}})) ]
\end{equation}
The right side is evaluated with the trigonometric identity:
\begin{equation}
 \sin(\alpha+\beta) = \sin(\alpha)\cos(\beta) + \cos(\alpha)\sin(\beta)
\end{equation}
Using:
\begin{equation}
 \sin(\arcsin(\frac{1}{\sqrt{1+x^2}})) = \frac{1}{\sqrt{1+x^2}}
\end{equation}
\begin{equation}
 \cos(\arcsin(x)) = \sqrt{1-x^2}
\end{equation}
\begin{equation}
 \cos(\arcsin(\frac{1}{\sqrt{1+x^2}})) = \frac{|x|}{\sqrt{1+x^2}}
\end{equation}
the right side yields:
\begin{equation}
 \frac{-n\,\sg(x)}{(1+x^2)^{\frac{n}{2}+1}} [ |x| \sin(n\arcsin(\frac{1}{\sqrt{1+x^2}})) + \cos(n\arcsin(\frac{1}{\sqrt{1+x^2}})) ]
\end{equation}
Because $\sg(x)|x|=x$ both sides are identical and the theorem is proved for $x\neq 0$.
For $x=0$ because $\sg(0)=1$ as in literature \cite{VL10} this theorem yields:
\begin{equation}
 D_x^n \arctan(x)|_{x=0} = (-1)^{n+1} (n-1)! \sin(n\frac{\pi}{2}) = 
 \begin{cases}
   (n-1)! (-1)^{\frac{n-1}{2}} & \text{if $n$ is odd} \\
   0 & \text{if $n$ is even} \\
 \end{cases}
\end{equation}
which is in agreement with theorem \ref{tander}.
\end{proof}

\section{The Function sg(x) for Complex x}

Let a complex $x$ be defined by:
\begin{equation}
 x = r e^{i\phi}
\end{equation}
where $r=|x|$ is its absolute value and $\phi=\arg(x)$ its angle with respect to the positive
real axis in the complex plane, and where $-\pi<\phi\leq\pi$.
Then the principal square root of a complex $x$ is defined by:
\begin{equation}
 \sqrt{x} = \sqrt{r}\, e^{i\phi/2}
\end{equation}
The theorems above lead to the following definition \cite{wiki}:
\begin{definition}
For complex $x$, let $\sqrt{x}$ be the principal square root of $x$, then:
\begin{equation}
 \sg(x) = 
 \begin{cases}
   \frac{\sqrt{x^2}}{x} = \frac{x}{\sqrt{x^2}} & \text{\rm if $x\neq 0$} \\
   1 & \text{\rm if $x=0$} \\
 \end{cases}
\end{equation}
The result of this definition is:
\begin{equation}
 \sg(x) = 
 \begin{cases}
   1 & \text{\rm if $\re(x) > 0$} \\
   -1 & \text{\rm if $\re(x) < 0$} \\
   1 & \text{\rm if $\re(x)=0$ and $\im(x)\geq 0$} \\
   -1 & \text{\rm if $\re(x)=0$ and $\im(x)< 0$} \\
 \end{cases}
\end{equation}
\end{definition}
For real $x$ this definition reduces to definition \ref{sgdef},
and from this definition follows that for real $x$: $\sg(ix)=\sg(x)$.
For complex $x$, $1/\sg(x)=\sg(x)$,
and for complex $x\neq 0$, $\sg(-x)=-\sg(x)$.
Replacing $x$ with $ix$, the following results for complex $x$:
\begin{equation}
 \sg(ix) = 
 \begin{cases}
   \frac{\sqrt{-x^2}}{ix} = \frac{ix}{\sqrt{-x^2}} & \text{\rm if $x\neq 0$} \\
   1 & \text{\rm if $x=0$} \\
 \end{cases}
\end{equation}
Replacing $x$ with $\sqrt{x}\sqrt{y}$ and using $(\sqrt{x})^2=x$, for complex $x$ and $y$:
\begin{equation}
 \sg(\sqrt{x}\sqrt{y}) = 
 \begin{cases}
   \frac{\sqrt{xy}}{\sqrt{x}\sqrt{y}} = \frac{\sqrt{x}\sqrt{y}}{\sqrt{xy}} & \text{\rm if $xy\neq 0$} \\
   1 & \text{\rm if $xy=0$} \\
 \end{cases}
\end{equation}
Taking $y=1$ it follows that for complex $x$: $\sg(\sqrt{x})=1$.
Replacing $x$ with $\sqrt{x}/\sqrt{y}$, for complex $x$ and $y\neq 0$:
\begin{equation}
 \sg(\frac{\sqrt{x}}{\sqrt{y}}) = 
 \begin{cases}
   \sqrt{\frac{x}{y}}\frac{\sqrt{y}}{\sqrt{x}} = \frac{\sqrt{x}}{\sqrt{y}}/\sqrt{\frac{x}{y}} & \text{\rm if $x\neq 0$} \\
   1 & \text{\rm if $x=0$} \\
 \end{cases}
\end{equation}
From these identities follows that for complex $x$:
$\sg(x)\sqrt{x^2}=x$, $\sg(x)x=\sqrt{x^2}$,
$\sg(ix)\sqrt{-x^2}=ix$, $\sg(ix)ix=\sqrt{-x^2}$,
and similarly for the other identities.\\
In \cite{K20} the following is proved for complex $x$:
\begin{equation}
 \arccot(x) + \arctan(x) = \frac{\pi}{2} \sg(x)
\end{equation}
This result is different from \cite{AS72} eq. 4.4.5 when $\re(x)=0$.\\
Substituting $\arctanh(x)=-i\arctan(ix)$ and $\arccoth(x)=i\,\arccot(ix)$:
\begin{equation}
 \arccoth(x) - \arctanh(x) = \frac{\pi}{2} i \, \sg(ix)
\end{equation}
The corresponding Mathematica\textsuperscript{\textregistered} \cite{W03} program:
\begin{alltt}
Sg[x_]:=If[Re[x]>0,1,If[Re[x]<0,-1,If[Im[x]>=0,1,-1]]]
\end{alltt}

\section{Examples}

\begin{equation}
 D_x \tan(x) = 1 + \tan^2(x)
\end{equation}
\begin{equation}
 D_x^2 \tan(x) = 2 \tan(x) + 2 \tan^3(x)
\end{equation}
\begin{equation}
 D_x^3 \tan(x) = 2 + 8 \tan^2(x) + 6 \tan^4(x)
\end{equation}
\begin{equation}
 D_x^4 \tan(x) = 16 \tan(x) + 40 \tan^3(x) + 24 \tan^5(x)
\end{equation}
\begin{equation}
 D_x^5 \tan(x) = 16 + 136 \tan^2(x) + 240 \tan^4(x) + 120 \tan^6(x)
\end{equation}
\begin{equation}
 D_x^6 \tan(x) = 272 \tan(x) + 1232 \tan^3(x) + 1680 \tan^5(x) + 720 \tan^7(x)
\end{equation}
\begin{equation}
 D_x \cot(x) = -1 - \cot^2(x)
\end{equation}
\begin{equation}
 D_x^2 \cot(x) = 2 \cot(x) + 2 \cot^3(x)
\end{equation}
\begin{equation}
 D_x^3 \cot(x) = -2 - 8 \cot^2(x) - 6 \cot^4(x)
\end{equation}
\begin{equation}
 D_x^4 \cot(x) = 16 \cot(x) + 40 \cot^3(x) + 24 \cot^5(x)
\end{equation}
\begin{equation}
 D_x^5 \cot(x) = -16 - 136 \cot^2(x) - 240 \cot^4(x) - 120 \cot^6(x)
\end{equation}
\begin{equation}
 D_x^6 \cot(x) = 272 \cot(x) + 1232 \cot^3(x) + 1680 \cot^5(x) + 720 \cot^7(x)
\end{equation}
\begin{equation}
 D_x \tanh(x) = 1 - \tanh^2(x)
\end{equation}
\begin{equation}
 D_x^2 \tanh(x) = -2 \tanh(x) + 2 \tanh^3(x)
\end{equation}
\begin{equation}
 D_x^3 \tanh(x) = -2 + 8 \tanh^2(x) - 6 \tanh^4(x)
\end{equation}
\begin{equation}
 D_x^4 \tanh(x) = 16 \tanh(x) - 40 \tanh^3(x) + 24 \tanh^5(x)
\end{equation}
\begin{equation}
 D_x^5 \tanh(x) = 16 - 136 \tanh^2(x) + 240 \tanh^4(x) - 120 \tanh^6(x)
\end{equation}
\begin{equation}
 D_x^6 \tanh(x) = -272 \tanh(x) + 1232 \tanh^3(x) - 1680 \tanh^5(x) + 720 \tanh^7(x)
\end{equation}
\begin{equation}
 D_x \coth(x) = 1 - \coth^2(x)
\end{equation}
\begin{equation}
 D_x^2 \coth(x) = -2 \coth(x) + 2 \coth^3(x)
\end{equation}
\begin{equation}
 D_x^3 \coth(x) = -2 + 8 \coth^2(x) - 6 \coth^4(x)
\end{equation}
\begin{equation}
 D_x^4 \coth(x) = 16 \coth(x) - 40 \coth^3(x) + 24 \coth^5(x)
\end{equation}
\begin{equation}
 D_x^5 \coth(x) = 16 - 136 \coth^2(x) + 240 \coth^4(x) - 120 \coth^6(x)
\end{equation}
\begin{equation}
 D_x^6 \coth(x) = -272 \coth(x) + 1232 \coth^3(x) - 1680 \coth^5(x) + 720 \coth^7(x)
\end{equation}
\begin{equation}
 D_x \arctan(x) = \frac{1}{1+x^2}
\end{equation}
\begin{equation}
 D_x^2 \arctan(x) = \frac{-2x}{(1+x^2)^2}
\end{equation}
\begin{equation}
 D_x^3 \arctan(x) = \frac{-2 ( 1 - 3x^2 )}{(1+x^2)^3}
\end{equation}
\begin{equation}
 D_x^4 \arctan(x) = \frac{24 x ( 1 - x^2 )}{(1+x^2)^4}
\end{equation}
\begin{equation}
 D_x^5 \arctan(x) = \frac{ 24 ( 1 - 10x^2 + 5x^4 )}{(1+x^2)^5}
\end{equation}
\begin{equation}
 D_x^6 \arctan(x) = \frac{ -240 x ( 3 - 10x^2 + 3x^4 )}{(1+x^2)^6}
\end{equation}
\begin{equation}
 D_x \arccot(x) = \frac{-1}{1+x^2}
\end{equation}
\begin{equation}
 D_x^2 \arccot(x) = \frac{2x}{(1+x^2)^2}
\end{equation}
\begin{equation}
 D_x^3 \arccot(x) = \frac{2 ( 1 - 3x^2 )}{(1+x^2)^3}
\end{equation}
\begin{equation}
 D_x^4 \arccot(x) = \frac{-24 x ( 1 - x^2 )}{(1+x^2)^4}
\end{equation}
\begin{equation}
 D_x^5 \arccot(x) = \frac{ -24 ( 1 - 10x^2 + 5x^4 )}{(1+x^2)^5}
\end{equation}
\begin{equation}
 D_x^6 \arccot(x) = \frac{ 240 x ( 3 - 10x^2 + 3x^4 )}{(1+x^2)^6}
\end{equation}
\begin{equation}
 D_x \arctanh(x) = \frac{1}{1-x^2}
\end{equation}
\begin{equation}
 D_x^2 \arctanh(x) = \frac{2x}{(1-x^2)^2}
\end{equation}
\begin{equation}
 D_x^3 \arctanh(x) = \frac{2(1+3x^2)}{(1-x^2)^3}
\end{equation}
\begin{equation}
 D_x^4 \arctanh(x) = \frac{24x(1+x^2)}{(1-x^2)^4}
\end{equation}
\begin{equation}
 D_x^5 \arctanh(x) = \frac{24(1+10x^2+5x^4)}{(1-x^2)^5}
\end{equation}
\begin{equation}
 D_x^6 \arctanh(x) = \frac{240x(3+10x^2+3x^4)}{(1-x^2)^6}
\end{equation}
\begin{equation}
 D_x \arccoth(x) = \frac{1}{1-x^2}
\end{equation}
\begin{equation}
 D_x^2 \arccoth(x) = \frac{2x}{(1-x^2)^2}
\end{equation}
\begin{equation}
 D_x^3 \arccoth(x) = \frac{2(1+3x^2)}{(1-x^2)^3}
\end{equation}
\begin{equation}
 D_x^4 \arccoth(x) = \frac{24x(1+x^2)}{(1-x^2)^4}
\end{equation}
\begin{equation}
 D_x^5 \arccoth(x) = \frac{24(1+10x^2+5x^4)}{(1-x^2)^5}
\end{equation}
\begin{equation}
 D_x^6 \arccoth(x) = \frac{240x(3+10x^2+3x^4)}{(1-x^2)^6}
\end{equation}

\pdfbookmark[0]{References}{}

\end{document}